\documentclass[10pt,reqno]{amsart}

\usepackage[a4paper,left=25mm,right=25mm,top=25mm,bottom=25mm,marginpar=25mm]{geometry}

\usepackage{amsfonts}
\usepackage{amsthm}
\usepackage{amssymb}
\usepackage{amsmath}
\usepackage{amscd}
\usepackage{bm}
\usepackage{mathtools}
\usepackage{color}
\usepackage{enumerate}
\usepackage{dsfont}
\usepackage[latin1]{inputenc}
\usepackage[displaymath,mathlines]{lineno}
\usepackage[mathscr]{eucal}
\numberwithin{equation}{section}
\usepackage{epstopdf} 
\usepackage{indentfirst}
\usepackage{lineno}
\usepackage[unicode=true]{hyperref}
\setpagewiselinenumbers
 
 \makeindex

\theoremstyle{plain}
\newtheorem{theorem}{Theorem}[section]

\newtheorem{proposition}[theorem]{Proposition}

\theoremstyle{definition}
\newtheorem{definition}[theorem]{Definition}

\DeclareMathOperator\supp{supp}

\title[Weak-strong uniqueness for Euler-Poisson systems]
      {The role of Riesz potentials in the weak-strong uniqueness for Euler-Poisson systems}

\author[Nuno J. Alves]{}


\email{nuno.januarioalves@kaust.edu.sa}

\begin{document}
\maketitle

\centerline{\scshape Nuno J. Alves}



\date{\today}

\begin{abstract}
In this article, the weak-strong uniqueness principle is proved for an Euler-Poisson system in the whole space, with initial data so that the strong solution exists. Some results on Riesz potentials are used to justify the considered weak formulation. Then, one follows the relative energy methodology and, in order to handle the solution of Poisson's equation, employs the theory of Riesz potentials.
\end{abstract}

\maketitle

\baselineskip 18pt

\section{Introduction}
This work is concerned with the weak-strong uniqueness principle for the following Euler-Poisson system in $\mathopen{]}0,T\mathclose{[}  \times \mathbb{R}^d,$ with $T < \infty$ and $d \in \mathbb{N}, \ d\geq 3:$
\begin{equation} \label{0EP}
 \begin{cases}
  \partial_t \rho + \nabla \cdot (\rho u) = 0 \\
  \partial_t (\rho u) + \nabla \cdot (\rho u \otimes u) + \nabla \rho^\gamma + \rho \nabla \phi = 0 \\\
  -\Delta \phi = \rho \ .
 \end{cases}
\end{equation} 
This system describes the evolution of a single-species fluid composed by charged carriers -- a basic model in plasma physics or semiconductor theory. The density of the fluid is represented by $\rho,$ its linear velocity is denoted by $u,$ and $\phi$ is the electric potential generated by the charged particles. The first equation is the continuity equation, which represents the conservation of mass or charge. The evolution of the momentum $\rho u$ is given by the second equation, with $\gamma > 1$ being the adiabatic exponent of the fluid. The last equation -- Poisson's equation -- determines the electric potential $\phi,$ which is represented by
\begin{equation} \label{phirepresent}
\phi(t,x)  = \frac{1}{c(d)}\int_{\mathbb{R}^d} \frac{\rho(t,y)}{|x-y|^{d-2}} \ dy \ ,
\end{equation}
where $c(d) = d(d-2)\mathcal{L}(B_1(0)) = \frac{d(d-2) \pi^{d/2}}{\Gamma(\frac{d}{2})+1}$ and $\mathcal{L}$ denotes the Lebesgue measure.  
Representation (\ref{phirepresent}) is a particular case of a more general concept -- Riesz potentials. \par 
The goal of this work is to establish the weak-strong uniqueness property of system (\ref{0EP}) using the machinery of Riesz potentials together with the relative energy method. The weak-strong uniqueness principle, with respect to certain classes of weak and strong solutions, says that if a weak solution coincides with the strong solution at the initial time, for initial data $(\rho_0,\rho_0 u_0, \phi_0)$ so that the strong solution exists, then they coincide for all times where both are defined. The notions of weak and strong solutions used for the current analysis are detailed in section \ref{sectionEP}. For those classes of solutions one obtains the following result:
\begin{theorem} \label{mainthm}
Let $(\rho, \rho u, \phi)$ be a dissipative weak solution of (\ref{0EP}) with $\gamma \geq \tfrac{2d}{d+1},$ and let $(\bar \rho, \bar \rho \bar u,\bar \phi)$ be a strong solution of (\ref{0EP}). If $(\rho_0, \rho_0 u_0) = (\bar \rho_0, \bar \rho_0 \bar u_0)$ then, for every $t \in [0,T[,$ 
\[\big(\rho(t), \rho(t) u(t), \phi(t)\big) = \big(\bar \rho(t), \bar \rho(t) \bar u(t), \bar \phi(t)\big) \ .\]
\end{theorem}
\par
The weak-strong uniqueness property lives behind the stability estimates on relaxation; see \cite{gasdynamics} for the weak-strong uniqueness of a Euler-Poisson system, for periodic solutions, as a consequence of its relaxation limit towards a Keller-Segel system. The weak-strong uniqueness principle has also been studied for the Euler-Poisson system with linear damping and confinement in a bounded smooth domain, for classes of measure-valued solutions; in this case being called measure-valued-strong uniqueness \cite{carrillomv}. Furthermore, it has as well been studied for a similar set of equations, the Navier-Stokes-Poisson system, for classes of weak solutions that are continuous in time with respect to the weak topology \cite{danica}. In those cases, the relative energy method was employed to achieve a stability estimate from which the weak-strong uniqueness follows. Here, that approach is also used for weak solutions that are continuous in time and defined in the whole space $\mathbb{R}^d$, which represents the novelty of the present work. Such weak solutions have not yet been proved to exist, and it is out of the scope of this manuscript to establish their existence. Nevertheless, one considers that notion of weak solution for its consistency with the a priori energy estimates of the system. The theory of Riesz potentials will be important in the justification of the weak formulation to what concerns the terms containing the electric potential $\phi$. The estimates on Riesz potentials also yield the thresholds for the adiabatic exponent $\gamma$. Those estimates are as well crucial to overcome the difficulties of the technical part of the proof of the main stability estimate, which is done via an integration by parts formula. The relative energy method is an efficient tool for the analysis of stability of systems of conservation laws; see \cite{dafermos} for an early work on the stability of thermoelastic fluids, and \cite{GLT} for more recent developments. This method has also been successful in establishing limiting processes; see \cite{gasdynamics} for the high-friction limit of single-species Euler flows towards gradient flows, and \cite{alves} for the relaxation limit of a two-species Euler-Poisson system with friction towards a bipolar drift-diffusion system.
\par 
Section \ref{sectionRP} presents all the results on Riesz potentials that will be important for the subsequent analysis. Those results are essentially a consequence of a classical result on Riesz potentials, Proposition \ref{Steinprop}, together with H\"{o}lder's inequality. \par
Section \ref{sectionEP} is devoted to considerations on the Euler-Poisson system (\ref{0EP}). First, one formally derives the energy and the relative energy identities for the system. The relative energy is the crucial concept for the analysis of the weak-strong uniqueness done here, as it serves as a yardstick for the comparison between the two solutions. Moreover, the notions of weak and strong solutions are given. Those will be used in the rigorous derivation of the relative energy inequality in the last section. 
\par
Section \ref{sectionwsu} is intended to establish the weak-strong uniqueness principle for system (\ref{0EP}) with respect to weak and strong solutions under the regularity specified in section \ref{sectionEP}. The crucial result, Theorem \ref{mainstabilitythm}, is an inequality satisfied by the relative energy from which Theorem \ref{mainthm} immediately follows. The most technical part of the proof of that result is handled using an integration by parts formula, Proposition \ref{intparts2prop}, which is proved using the results on Riesz potentials presented in subsection \ref{subsec22}.

\section{Riesz potentials} \label{sectionRP}

 Let $d \in \mathbb{N}$ and $0 < \alpha < d.$ The Riesz potential of degree $\alpha$ of a measurable function $f$ is the function $I_\alpha f$ given by 
 \[I_\alpha f(x) = \int_{\mathbb{R}^d} \frac{f(y)}{|x-y|^{d-\alpha}} \ dy \ .\] \\
Observe that the integrand part of (\ref{phirepresent}) coincides with the integrand part of the Riesz potential of degree $2$ of the density $\rho.$ To deal with the electric potential $\phi$ and with the electric field $-\nabla \phi$ one uses some integrability inequalities for Riesz potentials that will be derived here. The starting point to obtain those inequalities is the following classical result, proved in \cite{hedberg,stein}.
\vspace{2mm}

\begin{proposition} \label{Steinprop}
 Let $f \in L^p(\mathbb{R}^d),$ where $1 < p < d/\alpha$ with $0 < \alpha < d.$ Then
 \begin{equation} \label{Stein}
  ||I_\alpha f ||_{\frac{dp}{d-\alpha p}} \leq C(d,\alpha, p)\  ||f||_p \ .
 \end{equation}
\end{proposition}

\vspace{2mm}

\subsection{Preliminaries}
This subsection contains preliminary results that serve to emphasize the simple idea of choosing the integrability exponent $p$ in (\ref{Stein}) in such a way that $\frac{dp}{d-\alpha p}$ is the H\"{o}lder conjugate of some other integrability exponent $q$.
\begin{proposition}
 Let $d \in \mathbb{N}, \ 0 < \alpha, \beta < d$ and let $f \in L^p(\mathbb{R}^d) , \ g \in L^q(\mathbb{R}^d)$ with $p,q \in \ ]1,\infty[ .$
\begin{enumerate}[(i)]
\item If $1 < p,q < \frac d \alpha $ are such that $\frac{1}{p}+\frac{1}{q} = 1 + \frac{\alpha}{d},$ then 
\[||(I_\alpha f) g ||_1 \leq C(d,\alpha, p)\  ||f||_p ||g||_q \ . \]
\item If $1 < p,q < \frac{d}{\alpha + \beta} $ are such that $\frac{1}{p}+\frac{1}{q} = 1 + \frac{\alpha + \beta}{d},$ then 
 \[ ||(I_\alpha f)(I_\beta g) ||_1 \leq C(d,\alpha, \beta, p)\  ||f||_p ||g||_q \ . \]
\end{enumerate} 
 \end{proposition}
\begin{proof}~
\begin{enumerate}[(i)]
 \item First observe that 
 \[\dfrac{1}{p}+\dfrac{1}{q} = 1 + \dfrac{\alpha}{d} \ \Leftrightarrow \ q^\prime = \frac{q}{q-1} = \frac{dp}{d- \alpha p} \ , \]
 and \[1 < p \ \Leftrightarrow \ q < \frac{d}{\alpha} \ , \ \ \ p < \frac{d}{\alpha} \ \Leftrightarrow \ 1 < q \ . \]
 Thus, H\"{o}lder inequality together with (\ref{Stein}) yields
 \begin{equation*}
  \begin{split}
   ||(I_\alpha f) g ||_1 \ &\leq \ ||I_\alpha f ||_{q^\prime} \ ||g||_q \\ & =  \ ||I_\alpha f ||_{\frac{dp}{d- \alpha p}} \ ||g||_q \\
   & \leq \ C(d, \alpha, p) \ ||f||_p \ ||g||_q \ .
  \end{split}
\end{equation*}
\item Similarly,
 \[ \dfrac{1}{p}+\dfrac{1}{q} = 1 + \dfrac{\alpha + \beta}{d} \ \  \Leftrightarrow \ \ s = \dfrac{dp}{d-\alpha p} \ \text{and} \ s^\prime = \dfrac{s}{s-1} = \dfrac{dq}{d-\beta q} \ ,\]
and 
\[1 < p \ \Leftrightarrow \ q < \frac{d}{\alpha+\beta} \ , \ \ \ p < \frac{d}{\alpha+\beta} \ \Leftrightarrow \ 1 < q \ . \]
Thus, H\"{o}lder's inequality together with (\ref{Stein}) gives
\begin{equation*}
  \begin{split}
   ||(I_\alpha f) (I_\beta g) ||_1 \ &\leq \ ||I_\alpha f ||_{s} \ ||I_\beta g||_{s^\prime} \\ & =  \ ||I_\alpha f ||_{\frac{dp}{d- \alpha p}} \ ||I_\beta g||_{\frac{dq}{d- \beta q}} \\
   & \leq \ C(d, \alpha, p) \ C(d,\beta,q) \ ||f||_p \ ||g||_q \ .
  \end{split}
\end{equation*}
\end{enumerate}
\end{proof}

\subsection{Riesz potentials on \texorpdfstring{$\bm{L^1(\mathbb{R}^d) \cap L^\gamma(\mathbb{R}^d)}$}{L\unichar{"00B9}(R\unichar{"1D48})\unichar{"22C2}L\unichar{"1D5E}(R\unichar{"1D48})}} \label{subsec22}
The results presented here are obtained combining the same idea of the previous subsection with the interpolation of Lebesgue measurable spaces. These results will be crucial for the proof of Proposition \ref{intparts2prop}.
\begin{proposition}
 Let $d \in \mathbb{N}, \ 0 < \alpha \leq \beta < d/2$ and $\rho, \eta \in L^1(\mathbb{R}^d) \cap L^\gamma(\mathbb{R}^d)$ with $\gamma \geq \frac{2d}{d+2 \alpha} .$ Then 
\begin{equation} \label{Irho2}
||I_\alpha \rho||_2 \leq C(d,\alpha) \ ||\rho ||_{\frac{2d}{d+2 \alpha}} \ ,
\end{equation} 
\begin{equation} \label{IetaIrho1}
 ||(I_\beta \eta)(I_\alpha \rho)||_1 \leq C(d,\alpha,\beta) \ ||\eta ||_{\frac{2d}{d+2 \beta}} \ ||\rho ||_{\frac{2d}{d+2 \alpha}} \ .
 \end{equation}
\end{proposition}
\begin{proof}
 Set $p = \frac{2d}{d+2 \alpha}, \ q = \frac{2d}{d+2 \beta}$ and note that $\rho, \eta \in L^p(\mathbb{R}^d) \cap L^q(\mathbb{R}^d).$ Moreover, a simple calculation yields
$$2 = \frac{dp}{d-\alpha p} = \frac{dq}{d-\beta q} \ .$$
Then, (\ref{Irho2}) easily follows from (\ref{Stein}), and (\ref{IetaIrho1}) follows from (\ref{Irho2}) together with H\"{o}lder's inequality.
 \end{proof}

\begin{proposition}
 Let $d \in \mathbb{N}, \ 0 < \alpha < d$ and $\rho, \eta \in L^1(\mathbb{R}^d) \cap L^\gamma(\mathbb{R}^d)$ with $\gamma \geq \frac{2d}{d+ \alpha} .$ Then 
\begin{equation} \label{etaIrho1bound}
||\eta I_\alpha \rho||_1 \leq C(d,\alpha) \ ||\eta ||_{\frac{2d}{d+ \alpha}} ||\rho ||_{\frac{2d}{d+ \alpha}} \ .
\end{equation} 
\end{proposition}
\begin{proof}
 Set $\gamma_0 = \frac{2d}{d+ \alpha}$ and note that $\rho, \eta \in L^{\gamma_0}(\mathbb{R}^d).$ Moreover, a simple calculation yields
 $$\gamma_0^\prime = \frac{\gamma_0}{\gamma_0 - 1} = \frac{d \gamma_0}{d - \alpha \gamma_0} \ ,$$
 whence, by H\"{o}lder's inequality and (\ref{Stein}),
 $$||\eta I_\alpha \rho||_1 \leq  ||\eta ||_{\gamma_0} ||I_\alpha\rho ||_{\gamma_0^\prime} \leq C(d,\alpha) ||\eta ||_{\gamma_0} ||\rho ||_{\gamma_0} \ . $$
 
\end{proof}

\subsection{Riesz potentials on \texorpdfstring{$ \bm{C\big([0,T[;L^1(\mathbb{R}^d)\cap L^\gamma(\mathbb{R}^d)\big)}$}{C([0,T[;L^\unichar{"00B9}(R^\unichar{"1D48})\unichar{"22C2}L\unichar{"1D5E}(R\unichar{"1D48}))}}

As a consequence of the previous results, one has that the continuity in time of $\rho$ is propagated to $I_\alpha \rho$ and $\rho I_\alpha \rho.$ Precisely:

\begin{proposition} \label{rieszCL2}
 Let $d \in \mathbb{N}, \ 0 < \alpha < d/2$ and $\rho \in C\big([0,T[;L^1(\mathbb{R}^d) \cap L^\gamma(\mathbb{R}^d)\big)$ with $\gamma \geq \frac{2d}{d+2 \alpha} .$ Then 
 $I_\alpha \rho \in C\big([0,T[; L^2(\mathbb{R}^d)\big). $ 
\end{proposition}
\begin{proof}
 Set $p = \frac{2d}{d+2 \alpha}$ and observe that $\rho \in C\big([0,T[;L^p(\mathbb{R}^d) \big).$ Let $\varepsilon > 0, \ t \in [0,T[$ and set $\tilde{\varepsilon} = \frac \varepsilon  C$ where $C = C(d, \alpha)$ is as in (\ref{Irho2}). There exists $\delta = \delta(\varepsilon,t) > 0$  such that if $s \in [0,T[$ satisfies $|t-s| < \delta$ then $||\rho(t) - \rho(s) ||_p < \tilde{\varepsilon}.$ Thus, whenever $|t-s| < \delta$ one has 
 \begin{equation*}
 ||I_\alpha \rho(t)-I_\alpha \rho(s)  ||_2 = ||I_\alpha (\rho(t)- \rho(s) ) ||_2 \leq C ||\rho(t) - \rho(s) ||_p < C\tilde{\varepsilon}= \varepsilon \ ,
 \end{equation*}
 which completes the proof.
\end{proof}

\begin{proposition} \label{rhoIrhoCL1}
 Let $d \in \mathbb{N}, \ 0 < \alpha < d$ and $\rho \in C\big([0,T[;L^1(\mathbb{R}^d) \cap L^\gamma(\mathbb{R}^d)\big)$ with $\gamma \geq \frac{2d}{d+\alpha} .$ Then 
 $\rho I_\alpha \rho \in C\big([0,T[; L^1(\mathbb{R}^d)\big). $ 
\end{proposition}
\begin{proof}
 Set $\gamma_0 = \frac{2d}{d+\alpha}$ and note that $\rho \in C\big([0,T[;L^{\gamma_0}(\mathbb{R}^d) \big).$ Let $\varepsilon > 0, \ t \in [0,T[,$ and choose $\tilde{\varepsilon} > 0$ so that $C\tilde{\varepsilon}^2+2C||\rho(t)||_{\gamma_0} \tilde{\varepsilon} < \varepsilon,$ where $C=C(d,\alpha)$ is as in (\ref{etaIrho1bound}). There exists $\delta = \delta(\varepsilon,t)> 0$ such that whenever $s \in [0,T[$ satisfies $|t-s|<\delta$ one has $||\rho(t)-\rho(s) ||_{\gamma_0} < \tilde{\varepsilon}. $ Thus, using (\ref{etaIrho1bound}), for $|t-s|<\delta$ one has
\begin{equation*}
\begin{split}
||\rho(t)I_\alpha \rho(t) - \rho(s) I_\alpha \rho(s) ||_1  & \leq  ||(\rho(t)-\rho(s))I_\alpha \rho(t)||_1 + ||\rho(s)I_\alpha (\rho(t) - \rho(s))||_1\\
  & \leq C ||\rho(t) - \rho(s) ||_{\gamma_0} (||\rho(t)||_{\gamma_0}+||\rho(s)||_{\gamma_0} ) \\
  & \leq C ||\rho(t) - \rho(s) ||_{\gamma_0} (2||\rho(t)||_{\gamma_0}+||\rho(t) - \rho(s) ||_{\gamma_0}) \\
  & < C \tilde{\varepsilon} (2||\rho(t) ||_{\gamma_0}+\tilde{\varepsilon}) \\
  & < \varepsilon \ ,
\end{split}
\end{equation*} 
 which completes the proof.
\end{proof}

\section{Euler-Poisson system} \label{sectionEP}

Set $p(\rho) = \rho^\gamma,$ $\gamma > 1$ and let $h(\rho) = \tfrac{1}{\gamma - 1} \rho^\gamma.$ The functions $p$ and $h$ can be regarded as the pressure and internal energy functions of system (\ref{0EP}), respectively. It can be readily seen that $p$ and $h$ satisfy 
\begin{equation} \label{thermoident}
\rho h^{\prime \prime}(\rho)=p^\prime(\rho) \ , \ \ \ \rho h^\prime(\rho)=p(\rho)+h(\rho) \ \ \ \text{for} \ \rho > 0 \ .
\end{equation} 
Using the pressure function, system (\ref{0EP}) can be rewritten as:
\begin{equation} \label{EP}
 \begin{cases}
  \partial_t \rho + \nabla \cdot (\rho u) = 0 \\
  \partial_t (\rho u) + \nabla \cdot (\rho u \otimes u) + \nabla p(\rho) + \rho \nabla \phi = 0 \qquad \text{in} \ \mathopen{]}0,T\mathclose{[} \times \mathbb{R}^d\\
  -\Delta \phi = \rho \ 
 \end{cases}
\end{equation}
where $T>0$ is a finite fixed time horizon and $d \in \mathbb{N}, \ d \geq 3.$ Using the language of Riesz potentials, one regards the electric potential $\phi$ as $$\phi = \tfrac{1}{c(d)}I_2 \rho \ . $$ System (\ref{EP}) is supplemented with initial data $\rho_0, u_0,$ and $\phi_0 = \tfrac{1}{c(d)}I_2 \rho_0  .$
\subsection{Energy and relative energy}
Assume that $(\rho, \rho u)$ with $\phi = K * \rho$ is a smooth solution of $(\ref{EP}),$ where $K(x) = \frac{1}{c(d)|x|^{d-2}}.$ Multiplying the momentum equation by $u$ and using the continuity and Poisson equations yields
\begin{equation*}
\partial_t \big(\tfrac{1}{2} \rho |u|^2 + h(\rho) + \tfrac{1}{2}|\nabla \phi|^2 \big) + \nabla \cdot \big( \tfrac{1}{2} \rho |u|^2u+\rho h^\prime(\rho)+\phi \rho u - \phi \nabla \partial_t \phi \big) = 0 \ .
\end{equation*}
The previous identity represents the conservation of total energy of the system.
\par
The total energy of the system can be decomposed into kinetic and potential parts. The kinetic energy functional is the functional $\mathcal{K}=\mathcal{K}(\rho, \rho u)$ given by
\begin{equation} \label{kinfunct}
\mathcal{K}(\rho,\rho u) = \int_{\mathbb{R}^d} \tfrac{1}{2} \rho |u|^2 \ dx \ , 
\end{equation}
while the potential energy functional is the functional $\mathcal{E} = \mathcal{E}(\rho)$ given by 
  \begin{equation} \label{potfunct}
  \begin{split}
 \mathcal{E}(\rho) &= \int_{\mathbb{R}^d} h(\rho) + \tfrac{1}{2} \rho ( K * \rho) \ dx  \\
 &= \int_{\mathbb{R}^d} h(\rho) + \tfrac{1}{2}|\nabla \phi|^2 \ dx \ , \quad -\Delta \phi = \rho \ .
 \end{split}
\end{equation} \label{kineticderiv}
It is important to calculate the functional derivatives of these energy functionals. A straightforward computation gives
\begin{equation}
\frac{\delta \mathcal{K}}{\delta \rho}(\rho, \rho u) = -\tfrac{1}{2} |u|^2 \ , \quad \frac{\delta \mathcal{K}}{\delta (\rho u)}(\rho, \rho u) = u \ ,
\end{equation}
and from the symmetry of $K$ one also deduces that
\begin{equation} \label{potentialderiv}
\frac{\delta \mathcal{E}}{\delta \rho}(\rho) = h^\prime(\rho) + \phi \ .
\end{equation}
Thus, system (\ref{EP}) can be recast into a more abstract form: 
\begin{equation} \label{EPfunct}
 \begin{cases}
  \partial_t \rho + \nabla \cdot (\rho u) = 0 \\
  \partial_t (\rho u) + \nabla \cdot (\rho u \otimes u) + \rho \nabla \dfrac{\delta \mathcal{E}}{\delta \rho}(\rho) = 0 \ ,
 \end{cases}
\end{equation}
where $\mathcal{E}$ is given by (\ref{potfunct}). The last term on the right-hand side of the previous momentum equation can also be rewritten as 
\begin{equation}
\rho \nabla \dfrac{\delta \mathcal{E}}{\delta \rho}(\rho) = -\nabla \cdot S(\rho) \ ,
\end{equation}
where $S(\rho)=-p(\rho)I-\tfrac{1}{2}|\nabla \phi|^2I+\nabla \phi \otimes \nabla \phi$ is a stress tensor.
\par 
Under this abstract setup one replicates the calculations done in \cite{alves, GLT} to derive the relative energy identity for this system. Let $(\bar \rho, \bar \rho \bar u)$ with $\bar \phi = K * \bar \rho$ be another smooth solution of (\ref{EP}). The relative potential energy functional is given by the quadratic part of the Taylor series expansion of $\mathcal{E}.$ Precisely, 
\begin{equation} \label{relpotfunct}
\mathcal{E}(\rho | \bar \rho) = \mathcal{E}(\rho) - \mathcal{E}(\bar \rho) - \big\langle \frac{\delta \mathcal{E}}{\delta \rho}(\bar \rho), \rho - \bar \rho \big\rangle  = \int_{\mathbb{R}^d} h(\rho | \bar \rho) + \tfrac{1}{2}|\nabla (\phi - \bar \phi)|^2 \ dx \ ,
\end{equation}
where $h(\rho | \bar \rho) = h(\rho) - h(\bar \rho)- h^\prime(\bar \rho)(\rho - \bar \rho).$ Similarly, the relative kinetic energy functional is the functional $\mathcal{K}(\rho , \rho u | \bar \rho, \bar \rho \bar u)$ given by
\begin{equation*}
   \begin{split}
    \mathcal{K} (\rho,  \rho u  |  \bar{\rho}, \bar{\rho} \bar{u})  & =  \ \mathcal{K}(\rho, \rho u) - \mathcal{K}(\bar{\rho}, \bar{\rho} \bar{u})  - \big< \frac{\delta \mathcal{K}}{\delta \rho}(\bar{\rho}, \bar{\rho} \bar{u}) , \rho - \bar{\rho}\big> - \big< \frac{\delta \mathcal{K}}{\delta (\rho u)}(\bar{\rho}, \bar{\rho} \bar{u}) , \rho u - \bar{\rho} \bar u\big> \\
    & =   \int_{\mathbb{R}^d} \tfrac{1}{2} \rho |u - \bar{u}|^2  \ dx \ .
    \end{split}
 \end{equation*} 

Next one presents the evolution of the relative total energy of the system. For a detailed exposition of the calculations involved refer to \cite{alves, GLT}. 
\begin{equation} \label{formalRE}
\begin{split}
 \frac{d}{dt}  \Big( \mathcal{K}(\rho, \rho u | \bar \rho, \bar \rho \bar u) + \mathcal{E}(\rho | \bar{\rho}) \Big) = &  - \int_{\mathbb{R}^d} \nabla \bar{u} : \rho (u-\bar{u}) \otimes (u-\bar{u}) \ dx - \int_{\mathbb{R}^d} p(\rho | \bar{\rho}) \nabla \cdot \bar{u} \ dx \\
  & +  \int_{\mathbb{R}^d}  (\rho -\bar{\rho}) \bar{u}  \cdot \nabla(\phi- \bar{\phi}) \ dx \ .
\end{split}
\end{equation}

\subsection{Weak solutions} Before stating the notion of weak solution, one explains how the theory on Riesz potentials helps in giving meaning to the weak formulation of system (\ref{EP}), and derives some results concerning the potential $\phi = \tfrac{1}{c(d)}I_2 \rho$. \par 
Having in mind the momentum equation of system (\ref{EP}), a reasonable weak formulation requires that $p(\rho)$ and $\rho \nabla \phi$ belong to $L^1$ in space. The density $\rho$ is assumed to belong to $C\big([0,T[;  L^1(\mathbb{R}^d) \cap L^\gamma(\mathbb{R}^d) \big),$ from which it follows that $p(\rho)=\rho^\gamma \in C\big([0,T[;L^1(\mathbb{R}^d) \big).$ Having this regularity on the density $\rho$, using the machinery of Riesz potentials one will observe that for $\gamma \geq \tfrac{2d}{d+1}$ the term $\rho \nabla \phi$ belongs to $C\big([0,T[;L^1(\mathbb{R}^d, \mathbb{R}^d)\big).$ First, one needs the following result:

\begin{proposition}
 Let $\rho \in C\big([0,T[;L^1(\mathbb{R}^d) \cap L^\gamma(\mathbb{R}^d)\big)$ with $\gamma \geq \tfrac{2d}{d+2},$ and let $\phi =\frac{1}{c(d)}I_2 \rho.$ Then $\phi \in C\big([0,T[; L^{\frac{2d}{d-2}}(\mathbb{R}^d) \big)$ and its weak spatial gradient is given by 
 \begin{equation*} 
 \nabla \phi (t,x) = \frac{(2-d)}{c(d)} \int_{\mathbb{R}^d} \frac{(x-y)\rho(t,y)}{|x-y|^d} \ dy.
 \end{equation*}
 Furthermore, $\nabla \phi \in C\big([0,T[;L^2(\mathbb{R}^d,\mathbb{R}^d)\big) .$
 \end{proposition}
 \begin{proof}
 Set $p = \frac{2d}{d+2}$ and note that $\rho \in C\big([0,T[; L^p(\mathbb{R}^d) \big).$ For each $t \in [0,T[,$ using (\ref{Stein}) one has
 $$||\phi(t)||_{\frac{2d}{d-2}} = \tfrac{1}{c(d)}||I_2\rho(t) ||_{\frac{dp}{d-2p}} \leq C(d) \  ||\rho(t)||_p \ ,$$
 so $\phi(t) \in C\big([0,T[; L^{\frac{2d}{d-2}}(\mathbb{R}^d) \big).$ The continuity in time of $\phi$ follows from the linearity of $I_2.$ \par 
 Now let $v \in C_c^{\infty}(\mathbb{R}^d)$ and set $K(x,y) = \dfrac{1}{c(d)|x-y|^{d-2}}.$ One wishes to prove that for every $t \in [0,T[$ and $i = 1, \ldots, d $ it holds that
 \begin{equation} \label{weakgradient}
 \int_{\mathbb{R}^d} \phi(t) v_{x_i} \ dx = - \int_{\mathbb{R}^d} \int_{\mathbb{R}^d} K_{x_i}(x,y) \rho(t,y) \ dy \ v \ dx \ .
 \end{equation}
  Fix $t \in [0,T[$ and note that $\int_{\mathbb{R}^d} \phi(t) v_{x_i} \ dx < \infty $ since $\phi(t) \in L_{loc}^1(\mathbb{R}^d).$ Fubini's theorem then yields that
\[\int_{\mathbb{R}^d} \phi(t) v_{x_i} \ dx = \int_{\mathbb{R}^d} \int_{\mathbb{R}^d}  K(x,y) v_{x_i}(x)dx \ \rho(t,y) \ dy \ .\]
 Fix $y \in \mathbb{R}^d$ and let $\varepsilon > 0.$ Denote by $B_\varepsilon(y)$ the open ball of center $y$ and radius $\varepsilon.$ Splitting the inner integral and using the fact that $K( \cdot, y)$ is smooth in $B_\varepsilon^c (y)$ one obtains
\begin{equation}
\begin{split}
\int_{\mathbb{R}^d}  K(x,y) v_{x_i}(x) \ dx = & \int_{B_\varepsilon(y)}  K(x,y) v_{x_i}(x) \ dx + \int_{B_\varepsilon^c(y)}  K(x,y) v_{x_i}(x) \ dx \\
=&\int_{B_\varepsilon(y)}  K(x,y) v_{x_i}(x) \ dx + \int_{\partial B_\varepsilon(y)}  K(x,y) v(x) \nu_i(x) \ dS(x) \\
& - \int_{B_\varepsilon^c(y)}  K_{x_i}(x,y) v(x) \ dx \\
= & \ I_\varepsilon + J_\varepsilon - \int_{B_\varepsilon^c(y)}  K_{x_i}(x,y) v(x) \ dx \ ,
\end{split}
\end{equation} 
 where $\nu_i$ is the $i$-th component of the unit inward normal vector to $\partial B_\varepsilon(y).$ Next it is shown that $\lim I_\varepsilon = \lim J_\varepsilon = 0$ as $\varepsilon \to 0$ from which (\ref{weakgradient}) follows. Using polar coordinates and the fact that $v \in C^\infty_c(\mathbb{R}^d)$ one has that
\begin{equation*}
 \begin{split}
 |I_\varepsilon| & \leq \int_{B_\varepsilon(y)}|K(x,y)| |v_{x_i}(x)| \ dx \\
                 & \leq C \int_{B_\varepsilon(y)} \frac{1}{|x-y|^{d-2}} \ dx \\
                 & = C \int_0^\varepsilon \int_{\partial B_r(y)} \frac{1}{|x-y|^{d-2}} \ dS(x) dr \\
                & = C \varepsilon^2 \ ,         
 \end{split}
 \end{equation*}
and
 \begin{equation*}
 \begin{split}
|J_\varepsilon| & \leq \int_{\partial B_\varepsilon(y)}|K(x,y)| |v(x)| \ dS(x) \\
                 & \leq C \int_{\partial B_\varepsilon(y)} \frac{1}{|x-y|^{d-2}} \ dS(x) \\
                 & = C \varepsilon  \ .
 \end{split}
 \end{equation*}

Hence, the weak spatial gradient of $\phi$ is bounded from above by $I_1 |\rho|$
 \[|\nabla \phi |\leq \tfrac{d-2}{c(d)} I_1 |\rho| \ , \] whence, by Proposition \ref{rieszCL2} with $\alpha = 1$ it follows that $\nabla \phi \in C\big([0,T[;L^2(\mathbb{R}^d,\mathbb{R}^d)\big).$
 \end{proof} \noindent
Thus, taking $\alpha = 1$ in Proposition \ref{rhoIrhoCL1} one has that if $\gamma \geq \frac{2d}{d+1}$ then \[\rho \nabla \phi \in C\big([0,T[;L^1(\mathbb{R}^d, \mathbb{R}^d)\big) \ .\] \par 
Regarding the momentum $\rho u$, it is assumed to belong to $C\big([0,T[;L^1(\mathbb{R}^d,\mathbb{R}^d)\big),$ while $\sqrt{\rho} u$ is assumed to belong to $C\big([0,T[;L^2(\mathbb{R}^d,\mathbb{R}^d)\big).$ For a triple $(\rho, \rho u, \frac{1}{c(d)}I_2 \rho)$ with the regularity stated above, one defines the energy $ \mathcal{H} : [0,T[ \ \to \mathbb{R}$ by 
\[\mathcal{H}(t)= \int_{\mathbb{R}^d} \tfrac{1}{2} \rho(t) |u(t)|^2 + \tfrac{1}{\gamma-1} \rho^\gamma(t)+\tfrac{1}{2}|\nabla \phi(t)|^2 \ dx \ .\] Observe that in these conditions the energy $\mathcal{H}$ is a continuous function. \par
The precise definition of a weak solution for system (\ref{EP}) is now given. 
\begin{definition} \label{weakformulation}
 The vector function $(\rho,\rho u)$ with $\rho \geq 0$ and regularity
 \[\rho \in C\big([0,T[;  L^1(\mathbb{R}^d) \cap L^\gamma(\mathbb{R}^d) \big) \ ,\]
 \[\rho u \in C\big([0,T[;L^1(\mathbb{R}^d,\mathbb{R}^d)\big) \ ,\]
 \[\sqrt{\rho} u \in  C\big([0,T[;L^2(\mathbb{R}^d,\mathbb{R}^d)\big) \ , \]
together with $\phi = \frac{1}{c(d)}I_2 \rho, $ 
is a weak solution of (\ref{EP}) provided that:
\begin{enumerate}[(i)]
 \item $(\rho,\rho u)$ satisfies (\ref{EP}) in the weak sense
 \begin{equation} \label{weak1}
         -\int_0^T \int_{\mathbb{R}^d} \varphi_t \rho \ dxdt-\int_0^T \int_{\mathbb{R}^d} \nabla \varphi \cdot (\rho u) \ dxdt - \int_{\mathbb{R}^d} \varphi(0,x) \rho_0(x) \ dx=0 \ ,
        \end{equation} 
 \begin{equation} \label{weak2}
        \begin{split}
         - &  \int_0^T \int_{\mathbb{R}^d} \tilde{\varphi}_t \cdot (\rho u) \ dxdt -  \int_0^T \int_{\mathbb{R}^d} \nabla \tilde{\varphi} : \rho u \otimes u \ dx dt \\
        & -  \int_0^T \int_{\mathbb{R}^d} (\nabla \cdot \tilde{\varphi}) \rho^\gamma \ dxdt
        - \int_{\mathbb{R}^d} \tilde{\varphi}(0,x) \cdot \rho_0(x) u_0(x) \ dx \\
        = &   -  \int_0^T \int_{\mathbb{R}^d} \tilde{\varphi}\cdot (\rho \nabla \phi) \ dxdt \ ,
        \end{split}
        \end{equation} \\
for all Lipschitz test functions $\varphi: \mathopen{[}0,T\mathclose{[} \times \mathbb{R}^d \to \mathbb{R}, \ \tilde{\varphi}:\mathopen{[}0,T\mathclose{[} \times \mathbb{R}^d \to \mathbb{R}^d$ compactly supported in time, and with $\tilde{\varphi}$ satisfying, for each $t \in [0,T[,$ the following limiting behaviour at infinity 
$$\lim\limits_{|x| \to \infty} |\tilde{\varphi}(t,x)| = 0 \ ,$$
    \item Mass conservation:
    \begin{equation} \label{massconservation}
 ||\rho(t) ||_1 = M < \infty \ , \quad \forall t \in [0,T[ \ ,
  \end{equation} 
  \item Finitude of energy:
  \begin{equation} \label{finiteenergy}
  \underset{t \in [0,T[}{\text{sup}} \mathcal{H}(t) < \infty \ .
  \end{equation}
\end{enumerate}
Moreover, a weak solution of (\ref{EP}) is called dissipative if its energy $\mathcal{H}$ satisfies
 \begin{equation} \label{weakdissip}   
-  \int_0^T  \mathcal{H}(t) \dot{\theta}(t)dt 
    \leq   \mathcal{H}(0) \theta(0)
   \end{equation} 
  for any non-negative $\theta \in W^{1,\infty}([0,T[)$ with compact support.
\end{definition}

From a careful choice of the test function $\theta$ above one deduces that the energy $\mathcal{H}$ of a dissipative weak solution of (\ref{EP}) satisfies, for each $t \in [0,T[,$
\begin{equation} \label{weakdiss}
\mathcal{H}(t) \leq \mathcal{H}(0) \ .
\end{equation}
Indeed, fix $t \in [0,T[,$ let $\kappa$ be such that $t+\kappa < T$, and define $\theta : [0,T[\  \to \mathbb{R}$ by
\begin{equation} \label{thetatest}
\theta (\tau) = \theta_\kappa^t(\tau)=
\begin{cases}
1, \ \text{if} \ 0 \leq \tau < t \\
\frac{t- \tau}{\kappa}+1, \ \text{if} \ t \leq \tau < t+\kappa \\
0, \ \text{if} \ t+\kappa \leq \tau < T.
\end{cases}
\end{equation}
 Using this choice of $\theta$ in (\ref{weakdissip}) yields
 \begin{equation*}
 \frac{1}{\kappa} \int_t^{t+ \kappa} \mathcal{H}(\tau) d\tau  \leq  \mathcal{H}(0) \ .
 \end{equation*}
Letting $\kappa \to 0^+$ above one reaches the desired identity.

\subsection{Strong solutions}

In this subsection, the notion of strong solution that will be used in the subsequent analysis is described. \par
A Lipschitz vector function $(\bar \rho, \bar \rho \bar u),$ with regularity 
\[ \bar \rho \in L^\infty \big([0,T[; L^1(\mathbb{R}^d) \cap L^\infty(\mathbb{R}^d) \big) \ ,  \]
\[\bar u \in  L^\infty \big([0,T[; W^{1,1}(\mathbb{R}^d, \mathbb{R}^d) \cap W^{1,\infty}(\mathbb{R}^d,\mathbb{R}^d) \big)  \ , \]
together with $\bar \phi = \tfrac{1}{c(d)} I_2 \bar \rho,$ is called a strong solution of (\ref{EP}) if:

\begin{enumerate}[(i)]
\item $\bar \rho > 0,$

\item  $(\bar \rho, \bar \rho \bar u, \bar \phi)$ satisfies
\begin{equation} \label{EPstrong}
 \begin{cases}
  \partial_t \bar \rho + \nabla \cdot (\bar \rho \bar u) = 0 \\
  \partial_t \bar u + \bar u \cdot \nabla  \bar u + \nabla h^\prime(\bar \rho) +  \nabla \bar \phi = 0  
 \end{cases}
\end{equation}
 for almost every $(t,x) \in \  \mathopen{]}0,T\mathclose{[} \times \mathbb{R}^d,$ and  
 \[\lim\limits_{|x|\to \infty} |\bar u (t,x)| = 0 \quad \forall t \in [0,T[ \ ,\]
\item for each $t \in [0,T[$ and each $\kappa>0$ so that $t+\kappa<T$ the functions 
\[
\tau \mapsto \theta^t_\kappa(\tau) \frac{\delta \bar{\mathcal{H}}}{\delta \rho}(\tau)=\theta \big(-\tfrac{1}{2}|\bar u|^2 + h^\prime(\bar \rho) + \bar \phi \big) \ ,\]
\[\tau \mapsto \theta^t_\kappa(\tau) \frac{\delta \bar{\mathcal{H}}}{\delta(\rho u)}(\tau) = \theta \bar u \ ,\]
are Lipschitz continuous, where $\theta$ is given by (\ref{thetatest}) and $\bar{\mathcal{H}}: \mathopen{[}0,T\mathclose{[} \to \mathbb{R}$ is the total energy
\[\bar{\mathcal{H}}(t)= \int_{\mathbb{R}^d} \tfrac{1}{2} \bar \rho(t) |\bar u(t)|^2 + h( \bar  \rho(t))+\tfrac{1}{2}|\nabla \bar \phi(t)|^2   \ dx \ .\]
\end{enumerate} 
A strong solution is also assumed to emanate from initial data $(\bar \rho_0, \bar \rho_0 \bar u_0), \ \bar \phi_0 = \tfrac{1}{c(d)}I_2 \bar \rho_0$ that satisfy the bounds
\begin{equation} \label{stronginitialbounds}
\int_{\mathbb{R}^d} \bar \rho_0 \ dx = \bar M <  \infty \ , \quad \bar{\mathcal{H}}(0) <  \infty \ . 
\end{equation}
Multiplying the second equation of (\ref{EPstrong}) by $\bar \rho \bar u$ and integrating over space yields
$$\frac{d}{dt} \bar{\mathcal{H}} = 0 \ , $$
so $\bar{\mathcal{H}}(t) = \bar{\mathcal{H}}(0)$ for $t \in [0,T[.$ Thus, the continuity equation and the previous expression imply that $(\ref{stronginitialbounds})$ is propagated for all times $t \in \mathopen{[}0,T\mathclose{[}$.

\section{Weak-strong uniqueness} \label{sectionwsu}

For the notions of solutions specified in the previous section, one defines the relative energy function $\Psi:\mathopen{[}0,T\mathclose{[} \to \mathbb{R}$ by 
\begin{equation} \label{relenergyfunct}
\begin{split}
\Psi(t) &= \mathcal{H}(t) - \bar{\mathcal{H}}(t) - \big< \frac{\delta \bar{\mathcal{H}}}{\delta \rho}, \rho(t) - \bar \rho(t) \big> - \big< \frac{\delta \bar{\mathcal{H}}}{\delta (\rho u)}, \rho(t) u(t) - \bar \rho(t) \bar u(t) \big> \\
&= \int_{\mathbb{R}^d} \tfrac{1}{2}\rho(t) |u(t) - \bar u(t)|^2+h\big(\rho(t) | \bar \rho(t)\big)+\tfrac{1}{2} |\nabla \big(\phi(t) - \bar \phi(t)\big)|^2 \ dx \ .
\end{split}
\end{equation}

Due to the strict convexity of $h(\rho) = \tfrac{1}{\gamma - 1 }\rho^\gamma,$ one obtains the main result, Theorem \ref{mainthm}, as an immediate consequence of the next theorem. 

\begin{theorem} \label{mainstabilitythm}
Let $(\rho, \rho u)$ with $\phi = \tfrac{1}{c(d)} I_2 \rho$ be a dissipative weak solution of (\ref{EP}) with $\gamma \geq \tfrac{2d}{d+1},$ and let $(\bar \rho, \bar \rho \bar u)$ with $\bar \phi = \tfrac{1}{c(d)} I_2 \bar \rho$ be a strong solution of (\ref{EP}). There exists a positive constant $C$ such that the relative energy $\Psi$ of these solutions satisfies the stability estimate
\begin{equation} \label{mainstability}
\Psi(t) \leq e^{CT} \Psi(0)
\end{equation}
for $t \in [0,T[.$ Therefore, if $\Psi(0) = 0$ then $\Psi \equiv 0.$
\end{theorem}
The last two subsections of this part are devoted to give a proof of the previous theorem. 

\subsection{Integration by parts formulas}
Here, one presents two integration by parts formulas that will be essential for handling the terms containing the electric potential $\phi$. \par 
The first formula can be proved in a similar fashion as \cite[Propositon 4.2]{alves}.
\begin{proposition}
Let $\rho, \eta \in C\big([0,T[;L^1(\mathbb{R}^d) \cap L^\gamma(\mathbb{R}^d)\big)$ with $\gamma \geq \frac{2d}{d+2},$ and let $\phi = \frac{1}{c(d)}I_2 \rho,$ $\varphi = \frac{1}{c(d)}I_2 \eta.$ Then, for each $t \in [0,T[,$ 
 \begin{equation} \label{intparts}
 \int_{\mathbb{R}^d} \nabla \phi(t) \cdot \nabla \varphi(t) \ dx = \int_{\mathbb{R}^d} \rho(t) \varphi(t) \ dx = \int_{\mathbb{R}^d} \eta(t) \phi(t) \ dx  =  \tfrac{1}{c(d)} \int_{\mathbb{R}^d}  \int_{\mathbb{R}^d} \frac{\rho(t,x) \eta(t,y)}{|x-y|^{d-2}} \ dxdy \ .
\end{equation}
\end{proposition}
\vspace{2mm}
The second integration by parts formula relies on the following identity that is satisfied by classical solutions of $-\Delta \phi = \rho$ :
\[\rho \nabla \phi = \nabla \big( \tfrac{1}{2}|\nabla \phi|^2 \big) - \nabla \cdot (\nabla \phi \otimes \nabla \phi) \ . \]
\begin{proposition} \label{intparts2prop}
Let $\rho \in L^1(\mathbb{R}^d) \cap L^\gamma(\mathbb{R}^d)$, with $\gamma \geq \tfrac{2d}{d+1},$ let $\phi =\frac{1}{c(d)}I_2 \rho,$ and let $\bar u$ be a continuous vector field defined on $ \mathbb{R}^d$ satisfying $\lim\limits_{|x|\to \infty} |\bar u(x)| = 0$ and belonging to $W^{1,1}(\mathbb{R}^d,\mathbb{R}^d) \cap W^{1,\infty}(\mathbb{R}^d,\mathbb{R}^d).$ Then
\begin{equation} \label{intparts2}
\int_{\mathbb{R}^d} \rho \nabla \phi \cdot \bar u \ dx = \int_{\mathbb{R}^d} \nabla \bar u : \nabla \phi \otimes \nabla \phi \ dx - \int_{\mathbb{R}^d} (\nabla \cdot \bar u) \tfrac{1}{2}|\nabla \phi|^2 \ dx \ .
\end{equation}
\end{proposition}
\begin{proof}
Set $p = \tfrac{2d}{d+2}, \ \gamma_0 = \tfrac{2d}{d+1}$ and note that $\rho \in L^1(\mathbb{R}^d) \cap L^{\gamma_0}(\mathbb{R}^d).$ Therefore, by Proposition \ref{density}, there exists a sequence $(\rho_n)_{n \in \mathbb{N}} \subseteq C_c^\infty(\mathbb{R}^d)$ such that $\rho_n \to \rho$ in $L^1(\mathbb{R}^d) \cap L^{\gamma_0}(\mathbb{R}^d). $ For each $n \in \mathbb{N},$ let $\phi_n = \frac{1}{c(d)}I_2 \rho_n,$ and note that, as a consequence of the divergence theorem \cite{willem}, $(\rho_n,\phi_n)$ satisfies 
\[\int_{\mathbb{R}^d} \rho_n \nabla \phi_n \cdot \bar u \ dx = \int_{\mathbb{R}^d} \nabla \bar u : \nabla \phi_n \otimes \nabla \phi_n \ dx - \int_{\mathbb{R}^d} (\nabla \cdot \bar u) \tfrac{1}{2}|\nabla \phi_n|^2 \ dx \ . \]
Next, one wishes to pass to the limit in the above expression. Using (\ref{etaIrho1bound}) one deduces that 
\begin{equation*}
\begin{split}
\Big| \int_{\mathbb{R}^d} \rho \nabla \phi \cdot \bar u  \ dx - \int_{\mathbb{R}^d}\rho_n \nabla \phi_n \cdot \bar u \ dx \Big| & \leq ||\bar u||_\infty \int_{\mathbb{R}^d} |(\rho-\rho_n) \nabla \phi|+|\rho_n \nabla (\phi_n - \phi)| \ dx \\
& \leq C \big(||(\rho - \rho_n) I_1 \rho ||_1 + ||\rho_n I_1( \rho_n - \rho) ||_1 \big) \\
& \leq C \big(||\rho - \rho_n ||_{\gamma_0}||\rho ||_{\gamma_0}+||\rho_n ||_{\gamma_0}|| \rho_n-\rho||_{\gamma_0} \big) \\
& \to 0 \ \text{as} \ n \to \infty \ . 
\end{split}
\end{equation*}
Moreover, using (\ref{IetaIrho1}) together with the interpolation inequality one has that
\begin{equation*}
\begin{split}
\Big| \int_{\mathbb{R}^d} \nabla \bar u : \nabla \phi \otimes \nabla \phi  \ dx - \int_{\mathbb{R}^d} \nabla \bar u : \nabla \phi_n \otimes \nabla \phi_n \ dx \Big| & \leq ||\nabla \bar u||_\infty \int_{\mathbb{R}^d} |\nabla (\phi-\phi_n)||\nabla \phi|+|\nabla \phi_n||\nabla (\phi - \phi_n)| dx \\
& \leq C \big(||I_1(\rho - \rho_n) I_1 \rho ||_1 + ||(I_1 \rho_n) I_1( \rho_n - \rho) ||_1 \big) \\
& \leq C \big(||\rho - \rho_n ||_p||\rho ||_p+||\rho_n ||_p|| \rho_n-\rho||_p \big) \\
& \to 0 \ \text{as} \ n \to \infty \ ,
\end{split}
\end{equation*}
and 
\begin{equation*}
\begin{split}
\Big| \int_{\mathbb{R}^d} (\nabla \cdot \bar u) \tfrac{1}{2}|\nabla \phi|^2  \ dx - \int_{\mathbb{R}^d}(\nabla \cdot \bar u) \tfrac{1}{2}|\nabla \phi_n|^2 \ dx \Big| & \leq \frac{||\nabla \cdot \bar u||_\infty}{2} \int_{\mathbb{R}^d} \big| |\nabla \phi|^2-|\nabla \phi_n|^2 \big| \ dx \\
& \leq C ||I_1(\rho - \rho_n) I_1 (\rho+\rho_n) ||_1  \\
& \leq C ||\rho - \rho_n ||_p||\rho +\rho_n||_p \\
& \to 0 \ \text{as} \ n \to \infty \ .
\end{split}
\end{equation*}
 The result follows.
\end{proof}

\subsection{Relative energy inequality} 

The first step in order to obtain expression (\ref{mainstability}) is to derive a relative energy inequality satisfied by the relative energy $\Psi.$ The terms involving the electric potential $\phi$ are handled using the first integration by parts formula of the previous subsection. Precisely, one has:

\begin{proposition} \label{propREinequality}
Let $(\rho,\rho u),$ with $\phi = \tfrac{1}{c(d)}I_2\rho,$ be a dissipative weak solution of (\ref{EP}) with $\gamma \geq \frac{2d}{d+2}$, and let $(\bar{\rho},\bar{\rho} \bar{u}),$ with $\bar{\phi}= \tfrac{1}{c(d)}I_2\bar \rho,$ be a strong solution of (\ref{EP}). Then, for each $t \in [0,T[$, the relative energy $\Psi$ between these two solutions satisfies
\begin{equation} \label{REinequality}
\Psi(t)-\Psi(0)  \leq \mathcal{J}_1(t) + \mathcal{J}_2(t) + \mathcal{J}_3(t) \ ,
\end{equation}
where
\begin{equation*}
\begin{split}
\mathcal{J}_1(t) & = - \int_0^t \int_{\mathbb{R}^d} \nabla \bar{u}: \rho (u - \bar{u}) \otimes (u - \bar{u}) \ dxd \tau \ , \\
\mathcal{J}_2(t) & = -  \int_0^t \int_{\mathbb{R}^d} (\nabla \cdot \bar{u}) p(\rho | \bar{\rho}) \  dx d\tau \ , \\
\mathcal{J}_3(t) & =  \int_0^t \int_{\mathbb{R}^d} (\rho - \bar{\rho})\bar{u}\cdot \nabla (\phi - \bar{\phi}) \  dx d\tau \ . 
\end{split}
\end{equation*}
\end{proposition}
\begin{proof}
First recall that the energy a dissipative weak solution of (\ref{EP}) satisfies
\begin{equation} \label{weakREproof}
\mathcal{H}(t) \leq \mathcal{H}(0) \ ,
\end{equation} 
 while the energy of a strong solution of (\ref{EP}) satisfies 
 \begin{equation} \label{strongREproof}
 \bar{ \mathcal{H}}(t) = \bar{\mathcal{H}}(0) \ ,
 \end{equation} 
for every $t \in [0,T[.$ Next, one considers the difference $(\rho - \bar{\rho},\rho u - \bar{\rho} \bar{u}) $ between a weak and a strong solutions of (\ref{EP}):
\begin{equation*}
         -\int_0^T \int_{\mathbb{R}^d} \partial_t \varphi (\rho - \bar{\rho}) \ dxdt-\int_0^T \int_{\mathbb{R}^d} \nabla \varphi \cdot (\rho u - \bar{\rho} \bar{u}) \ dxdt - \int_{\mathbb{R}^d} \varphi (\rho - \bar{\rho}) \big|_{t=0} \ dx=0 \ ,
\end{equation*} 
\begin{equation*}
        \begin{split}
        -&  \int_0^T \int_{\mathbb{R}^d} \partial_t \tilde{\varphi} \cdot (\rho u- \bar{\rho} \bar{u}) \ dxdt -  \int_0^T \int_{\mathbb{R}^d} \nabla \tilde{\varphi} : (\rho u \otimes u - \bar{\rho} \bar{u} \otimes \bar{u}) \ dx dt \\
        &-  \int_0^T \int_{\mathbb{R}^d} (\nabla \cdot \tilde{\varphi}) \big(p(\rho)-p(\bar{\rho}) \big) \ dxdt
        - \int_{\mathbb{R}^d} \tilde{\varphi} \cdot (\rho u - \bar{\rho} \bar{u})\big|_{t=0} \ dx \\
        =& \  -  \int_0^T \int_{\mathbb{R}^d} \tilde{\varphi} \cdot (\rho \nabla \phi - \bar{\rho} \nabla \bar{\phi}) \ dxdt \ ,
        \end{split}
\end{equation*} 
for all Lipschitz test functions $\varphi: [0,T[ \times \mathbb{R}^d \to \mathbb{R}, \ \tilde{\varphi}:[0,T[ \times \mathbb{R}^d \to \mathbb{R}^d$ compactly supported in time and having for each $t \in [0,T[$ the following limiting behaviour at infinity 
\[\lim\limits_{|x| \to \infty} |\tilde{\varphi}(t,x)| = 0 \ .\]
Using as test functions the functions $\varphi$ and $\tilde{\varphi}$ given by 
\[\varphi = \theta \frac{\delta \bar{\mathcal{H}}}{\delta \rho} \ , \ \ \ \tilde{\varphi}=\theta \frac{\delta \bar{\mathcal{H}}}{\delta(\rho u)} \ ,\]
where $\theta$ is given by (\ref{thetatest}), after letting $\kappa \to 0^+$ one obtains

\begin{equation} \label{weakdiff1}
    \begin{split}
      & \int_{\mathbb{R}^d} \frac{\delta \bar{\mathcal{H}}}{\delta \rho}(\rho - \bar{\rho})\Big|_{\tau = 0}^{\tau = t} \ dx  -\int_0^t \int_{\mathbb{R}^d} \partial_\tau \frac{\delta \bar{\mathcal{H}}}{\delta \rho}(\rho - \bar{\rho}) \ dxd\tau -  \int_0^t \int_{\mathbb{R}^d} \nabla \frac{\delta \bar{\mathcal{H}}}{\delta \rho} \cdot (\rho u - \bar{\rho} \bar{u}) \ dxd\tau =  0 \ ,
    \end{split}
\end{equation}
\begin{equation} \label{weakdiff2}
    \begin{split}
       & \int_{\mathbb{R}^d} \big( \frac{\delta \bar{\mathcal{H}}}{\delta(\rho u)} \cdot (\rho u - \bar{\rho} \bar{u}) \big) \Big|_{\tau = 0}^{\tau = t} \ dx -  \int_0^t \int_{\mathbb{R}^d} \partial_\tau \frac{\delta \bar{\mathcal{H}}}{\delta(\rho u)} \cdot (\rho u - \bar{\rho} \bar{u}) \ dxd\tau \\
       & -   \int_0^t \int_{\mathbb{R}^d} \nabla \frac{\delta \bar{\mathcal{H}}}{\delta(\rho u)} : (\rho u \otimes u - \bar{\rho} \bar{u} \otimes \bar{u}) \ dxd\tau- \int_0^t \int_{\mathbb{R}^d} \nabla \cdot \frac{\delta \bar{\mathcal{H}}}{\delta(\rho u)}\big(p(\rho)-p(\bar{\rho})\big) \ dxd\tau \\
        = & - \int_0^t \int_{\mathbb{R}^d} \frac{\delta \bar{\mathcal{H}}}{\delta(\rho u)} \cdot (\rho \nabla \phi - \bar{\rho} \nabla \bar{\phi}) \ dxd\tau \ .
  \end{split}
\end{equation} 
The desired identity is then obtained by subtracting (\ref{strongREproof}), (\ref{weakdiff1}) and (\ref{weakdiff2}) from (\ref{weakREproof}), together with a meticulous rearrangement of the terms that is done using the continuity equation $\partial_t \bar{\rho}+\nabla \cdot (\bar{\rho} \bar{u}) = 0,$ the equation $\partial_t \bar u + \bar u \cdot \nabla \bar u = -\nabla \big(h^\prime(\bar \rho) + \bar \phi \big),$ and the second equality of (\ref{intparts}). For a detailed presentation of the calculations involved refer to \cite{alves, GLT, gasdynamics}.
\end{proof}
\subsection{Bounds in terms of the relative energy} 
In this section one shows that 
\begin{equation}
\mathcal{J}_i(t) \leq C \int_0^t \Psi(\tau) \ d\tau \ , \quad i=1,2,3 \ ,
\end{equation}
for some positive constant $C$ depending only on the dimension $d,$ on the adiabatic exponent $\gamma$ and on the $L^\infty$ norm of $\bar u$ and its derivatives, under the same hypothesis of Theorem \ref{mainstabilitythm}. Then Gronwall's inequality will imply (\ref{mainstability}). \par The first two terms are treated as follows,
\begin{equation*}
  \begin{split}
    \mathcal{J}_1(t) &=  - \int_0^t \int_{\mathbb{R}^d} \nabla \bar{u}: \rho (u - \bar{u}) \otimes (u - \bar{u}) \ dxd \tau \\
   & \leq  C \ ||\nabla \bar{u} ||_{\infty} \int_0^t \int_{\mathbb{R}^d}  \rho |u - \bar{u}|^2 \ dxd\tau \\
   & \leq C \int_0^t \Psi(\tau) \ d\tau \ ,
  \end{split}
\end{equation*} \par 
\begin{equation*}
  \begin{split}
    \mathcal{J}_2(t) &=  - \int_0^t \int_{\mathbb{R}^d} (\nabla \cdot \bar{u}) p(\rho | \bar{\rho}) \ dxd \tau \\
   &  \leq ||\nabla \cdot \bar{u} ||_{\infty} (\gamma - 1) \int_0^t \int_{\mathbb{R}^d} h(\rho | \bar{\rho}) \ dxd\tau \\
   & \leq C \int_0^t \Psi(\tau) \ d\tau \ ,
  \end{split}
\end{equation*} \par 
whereas for the third term one uses the integration by parts formula (\ref{intparts2}) to deduce that
\begin{equation*}
  \begin{split}
    \mathcal{J}_3(t) &=  \int_0^t \int_{\mathbb{R}^d} (\rho - \bar{\rho})\bar{u}\cdot \nabla (\phi - \bar{\phi}) \ dx d\tau \\ 
    & =     \int_0^t  \int_{\mathbb{R}^d} \nabla \bar u : \nabla (\phi - \bar \phi) \otimes \nabla (\phi  - \bar \phi) \ dx d\tau - \int_0^t \int_{\mathbb{R}^d} (\nabla \cdot \bar u) \tfrac{1}{2}|\nabla (\phi-\bar \phi)|^2 \ dx  d\tau \\
    & \leq C \ (||\nabla \bar u ||_\infty+ ||\nabla \cdot \bar u ||_\infty) \int_0^t \int_{\mathbb{R}^d}  |\nabla (\phi-\bar \phi)|^2 \ dx  d\tau \\
    & \leq C \int_0^t \Psi(\tau) \ d\tau \ ,
  \end{split}
\end{equation*}
as desired.
\appendix
\section*{Appendix}
\setcounter{equation}{0}
\setcounter{theorem}{0}
\renewcommand\thesection{A}
Here it is proved that $C_c^\infty(\mathbb{R}^d))$ is dense in $L^1(\mathbb{R}^d) \cap L^\gamma(\mathbb{R}^d).$ This result is used in the proof of Proposition \ref{intparts2prop}.
\begin{proposition} \label{density}
 Let $f \in L^1(\mathbb{R}^d) \cap L^\gamma(\mathbb{R}^d)$ with $\gamma > 1.$ Given $\varepsilon > 0$ there exists $\varphi \in C^\infty_c(\mathbb{R}^d)$ such that 
 $$||f-\varphi||_1+||f-\varphi||_\gamma < \varepsilon \ .$$
\end{proposition}
\begin{proof}
 Let $f \in L^1(\mathbb{R}^d) \cap L^\gamma(\mathbb{R}^d)$ and $\varepsilon > 0 .$ Any measurable function can be written as the difference between two non-negative measurable functions, so one can assume that $f \geq 0.$  Set $$S = \{s:\mathbb{R}^d \to \mathbb{R} \ \text{measurable simple function} \ | \ \mathcal{L}(\{ x \in \mathbb{R}^d \ | \ s(x) \neq 0 \}) < \infty \} \ . $$
 Given that $f$ is measurable, there exists a sequence of measurable simple functions $(s_n)_{n \in \mathbb{N}}$ such that $0 \leq s_1 \leq s_2 \leq \ldots \leq f$ and $s_n(x) \to f(x) \ \forall x \in \mathbb{R}^d.$ Since $f \in L^1(\mathbb{R}^d) \cap L^\gamma(\mathbb{R}^d),$ then $(s_n)_{n \in \mathbb{N}} \subseteq L^1(\mathbb{R}^d) \cap L^\gamma(\mathbb{R}^d)$ whence
 $(s_n)_{n \in \mathbb{N}} \subseteq S.$ Let $h_n = |f-s_n|+|f-s_n|^\gamma$ and observe that $|h_n| \leq 2|f|+2^\gamma |f|^\gamma \in L^1(\mathbb{R}^d),$ $h_n(x) \to 0 \ \forall x \in \mathbb{R}^d.$ Consequently, by Lebesgue's dominated convergence theorem,
 $$||f-s_n ||_1 + ||f-s_n ||_\gamma^\gamma = \int_{\mathbb{R}^d}h_n \ dx \to 0 \ . $$
 Hence, there exists $s \in S$ such that 
 $$ ||f-s ||_1 + ||f-s||_\gamma < \frac{\varepsilon}{4} \ .$$ Let $\tilde{\varepsilon}>0$ be such that $2||s||_\infty(\tilde{\varepsilon}+\tilde{\varepsilon}^{1/\gamma}) < \varepsilon/4 \ .$
 By Lusin's theorem there exists $g \in C_c(\mathbb{R}^d)$ such that $\mathcal{L}(\{x \in \mathbb{R}^d \ | \ g(x) \neq s(x) \}) < \tilde{\varepsilon}$ and $|g| \leq ||s||_\infty \ .$ Thus,
 \begin{equation*}
  \begin{split}
   ||g-s||_1+||g-s||_\gamma & = \int_{\mathbb{R}^d}|g-s| \ dx+\Big(\int_{\mathbb{R}^d}|g-s|^\gamma  \ dx \Big)^{1/\gamma} \\
   & = \int_{\{g \neq s\}}|g-s| \ dx+\Big(\int_{\{g \neq s\}}|g-s|^\gamma  \ dx \Big)^{1/\gamma} \\
   & < 2\tilde{\varepsilon}||s||_\infty+2\tilde{\varepsilon}^{1/\gamma}||s||_\infty \\
   & < \frac{\varepsilon}{4} \ ,
   \end{split}
\end{equation*}
so 
\begin{equation*}
  \begin{split}
   ||f-g||_1+||f-g||_\gamma & \leq  ||f-s||_1+||s-g||_1 + ||f-s||_\gamma+||s-g||_\gamma\\
   & < \varepsilon/4+\varepsilon/4 \\
   & = \frac{\varepsilon}{2} \ . 
 \end{split}
\end{equation*}
Now let $\eta$ be the standard mollifier and set $g_\delta = \eta_\delta * g,$ for $0 < \delta < 1 \ . $ Since $g$ is continuous, $g_\delta \to g$ uniformly on all compact subsets of $\mathbb{R}^d,$ as $\delta \to 0.$ Moreover, $g$ has compact support so there exists $r > 0$ such that $\supp(g) \ \subseteq \ B_r(0),$ hence 
\[\supp(g_\delta) \  \subseteq \ \supp(g)+\overline{B_\delta(0)} \subseteq B_r(0)+\overline{B_\delta(0)} \  \subseteq \ \overline{B_{r+\delta}(0)} \ \subseteq \ \overline{B_{r+1}(0)} \eqqcolon K \ .\]
Observe that 
\begin{equation*}
  \begin{split}
   ||g_\delta-g||_1+||g_\delta-g||_\gamma & =  \int_{\mathbb{R}^d}|g_\delta-g| \ dx+\Big(\int_{\mathbb{R}^d}|g_\delta-g|^\gamma \ dx \Big)^{1/\gamma} \\
   & = \int_K |g_\delta-g| \ dx+\Big(\int_K |g_\delta-g|^\gamma \ dx \Big)^{1/\gamma} \\
   & \leq \sup\limits_K |g_\delta-g| \big(\mathcal{L}(K)+\mathcal{L}(K)^{1/\gamma}\big) \ .
    \end{split}
\end{equation*}
Given that $\sup\limits_K |g_\delta-g| \to 0$ as $\delta \to 0,$ there exists $0 < \tilde{\delta} < 1$ such that 
\[\sup\limits_K |g_\delta-g| < \frac{\varepsilon}{2 \big(\mathcal{L}(K)+\mathcal{L}(K)^{1/\gamma} \big)} \ \ \ \text{whenever} \ \delta < \tilde{\delta} \ .\]
Let $\delta_0 = \tilde{\delta}/2$ and set $\varphi \coloneqq g_{\delta_0} \in C_c^\infty(\mathbb{R}^d).$ Finally, one gets
\begin{equation*}
  \begin{split}
   ||f-\varphi||_1+||f-\varphi||_\gamma & \leq  ||f-g||_1+||g-\varphi||_1 + ||f-g||_\gamma+||g-\varphi||_\gamma\\
   & < \varepsilon/2+\varepsilon/2 \\
   & = \varepsilon \ ,
 \end{split}
\end{equation*}
as desired.
\end{proof}

\section*{Acknowledgments}
The author would like to thank Professor Athanasios Tzavaras for many helpful conversations.

\end{document}